\documentclass{article}
\usepackage{cite,amsmath,amsfonts,amssymb,amsthm,enumerate}

\newtheorem{theorem}{Theorem}[section]
\newtheorem{proposition}[theorem]{Proposition}
\newtheorem{lemma}[theorem]{Lemma}
\newtheorem{corollary}[theorem]{Corollary}
\theoremstyle{definition}
\newtheorem{definition}[theorem]{Definition}

\providecommand{\cF}{\mathcal{F}}
\providecommand{\ZZ}{\mathbb{Z}}
\providecommand{\comp}[1]{\overline{#1}}
\DeclareMathOperator{\agr}{\nabla}
\DeclareMathOperator{\symdiff}{\Delta}
\providecommand{\xor}{\oplus}
\providecommand{\bigxor}{\bigoplus}
\providecommand{\inner}[2]{\langle #1, #2 \rangle}
\providecommand{\orth}[1]{#1^\perp}
\providecommand{\Kvecsp}[1]{\mathcal{K}_#1}
\providecommand{\tr}[1]{#1^T}

\begin{document}
\title{Triangle-intersecting families on eight vertices}
\author{Yuval Filmus\footnote{University of Toronto, \texttt{yuvalf@cs.toronto.edu}. Supported by NSERC.}}
\maketitle

\begin{abstract}
Simonovits and S\'{o}s~\cite{sos} conjectured that the maximal size of a triangle-intersecting family of graphs on $n$ vertices is $2^{\binom{n}{2}-3}$. Their conjecture has recently been proved~\cite{eff} using spectral methods. We provide an elementary proof of the special case of $8$ vertices using a partition argument.
\end{abstract}

\section{Introduction} \label{sec:intro}
The seminal paper of Erd\H{o}s, Ko and Rado~\cite{EKR} has initiated the study of intersecting families in extremal combinatorics. The primary object of investigation has been intersecting families of sets~\cite{EKR,Katona,ak}. Other authors considered more structured sets: for example, Deza and Frankl~\cite{df} considered intersecting families of permutations, and Simonovits and S\'{o}s~\cite{sos} considered graphical intersecting families.

One of the problems Simonovits and S\'{o}s considered is triangle-intersecting families. A family $\cF$ of subgraphs of $K_n$ is \emph{triangle-intersecting} if the intersection of any two $G_1,G_2 \in \cF$ contains a triangle. One example of such a family is a \emph{kernel system}, which consists of all graphs containing some fixed triangle. A kernel system contains $2^{\binom{n}{2}-3}$ graphs.

Simonovits and S\'{o}s conjectured that kernel systems are the unique maximal triangle-intersecting families. Their conjecture has recently been proved using spectral methods~\cite{eff}. In this note, we present a much simpler proof for the case $n=8$. We have been unable to extend our methods beyond $n=8$.

\medskip

Chung et al.~\cite{cgfs} provided the first non-trivial upper bound $2^{\binom{n}{2}-2}$ on the size of a triangle-intersecting family, using entropy methods. They also showed that the same bound holds even if we relax the definition by changing intersection $G_1 \cap G_2$ to agreement $G_1 \agr G_2 = (G_1 \cap G_2) \cup (\comp{G_1} \cap \comp{G_2})$. The latter is proved by reducing the problem of triangle-agreeing families to the original problem of triangle-intersecting families, and it applies in many other settings.

Suprisingly, all the proofs mentioned above actually work for even for \emph{non-bipartite-agreeing} families. Those are families in which we only require that the agreement of any two graphs be non-bipartite. This fact could perhaps be traced back to Tur\'{a}n's theorem on maximal triangle-free graphs.

\section{Notation} \label{sec:notation}
In this section we gather some notations which will be used in the sequel.

We will consider graphs as sets of their edges (at any given time, the set of vertices will be fixed). Therefore $G_1 \cap G_2$ is the graph whose edge set is the intersection of the edge sets of $G_1$ and $G_2$.

We use the following notation for common set operations:
\begin{itemize}
 \item The complement $\comp{G}$ of a graph $G$ is obtained by complementing the edge set with respect to the relevant complete graph.
 \item The symmetric difference $G_1 \symdiff G_2$ is defined by
\[ G_1 \symdiff G_2 = (G_1 \setminus G_2) \cup (G_2 \setminus G_1). \]
 \item The agreement $G_1 \agr G_2$ is defined by
\[ G_1 \agr G_2 = \comp{G_1 \symdiff G_2} = (G_1 \cap G_2) \cup (\comp{G_1} \cap \comp{G_2}). \]
\end{itemize}
The agreement operator outputs the set of elements on which both inputs agree.

The subgraphs of $K_n$ under the operation $\symdiff$ form a vector space over $\ZZ_2$, which we denote by $\Kvecsp{n}$.

\medskip

In the sequel we will discuss families of subgraphs of $K_n$. Such a family will be called a \emph{family of graphs on $n$ vertices}. We will be interested in the following types of families:
\begin{itemize}
 \item Triangle-intersecting families: the intersection of any two graphs contains a triangle.
 \item Non-bipartite-intersecting families: the intersection of any two graphs is not bipartite.
 \item Triangle-agreeing families: the agreement of any two graphs contains a triangle.
 \item Non-bipartite-agreeing families: the agreement of any two graphs is not bipartite.
\end{itemize}
Note that non-bipartite-intersecting families are also triangle-intersecting, and that X-agreeing families are also X-intersecting.

As mentioned in the introduction, the optimal families are kernel systems, to use the nomenclature of~\cite{cgfs}. The relevant kernel systems will be the following:
\begin{itemize}
 \item Triangle-junta: a family consisting of all supersets of some fixed triangle.
 \item Triangulumvirate: a family of the form $\{G : G \cap T = T_0\}$, where $T$ is a triangle and $T_0 \subseteq T$.
\end{itemize}
Notice that a triangle-junta is non-bipartite-intersecting, and that a triangulumvirate is non-bipartite-agreeing.

\medskip

In sections~\ref{sec:proof} and~\ref{sec:antilinear} we consider permutations in $S_8$. We think of these permutations as acting on the group $\ZZ_2^3$. We represent the elements of this group as numbers $\{0,\ldots,7\}$. The group operation then corresponds to exclusive or ($\xor$).

The inner product on elements from $\ZZ_2^n$ is denoted $\inner{\cdot}{\cdot}$. The set of all elements from $\ZZ_2^n$ orthogonal to some non-zero $x \in \ZZ_2^n$ is denoted $\orth{x}$.

\section{The proof} \label{sec:proof}
Our goal in this section is to prove the following theorem.
\begin{theorem} \label{thm:main}
 If $\cF$ is a non-bipartite-agreeing family of graphs on $8$ vertices then
\[ |\cF| \leq 2^{\binom{n}{2} - 3}, \]
with equality if and only if $\cF$ is a triangulumvirate.
\end{theorem}
We remind the reader that this result is proved for any number of vertices in~\cite{eff}.

Our proof proceeds along the following steps:
\begin{enumerate}
 \item We construct a three-dimensional subspace $V$ of $\Kvecsp{8}$ whose non-zero vectors are complements of cubes.
 \item A non-bipartite-agreeing family can intersect any coset of $V$ in at most one vector, hence the upper bound.
 \item For non-bipartite-intersecting families, a computer search shows that the unique maximal families are triangle-juntas; the computer search utilizes the fact that a maximal family must intersect each coset of $V$ in exactly one vector.
 \item Uniqueness for non-bipartite-agreeing families follows by a transference argument from~\cite{eff}.
\end{enumerate}

We begin by constructing $V$. The construction hinges upon a special type of permutation in $S_8$ we call \emph{antilinear}.
\begin{definition} \label{def:antilinear}
A permutation $\pi \in S_8$ is \emph{antilinear} if:
\begin{enumerate}[(a)]
 \item $\pi(0) = 0$.
 \item For non-zero $x,y,z \in \ZZ_2^3$, either $x \oplus y \oplus z \neq 0$ or $\pi(x) \oplus \pi(y) \oplus \pi(z) \neq 0$.
\end{enumerate}
\end{definition}
It is clear from the symmetry of the definition that $\pi$ is antilinear if and only if its inverse $\pi^{-1}$ is antilinear.

Antilinear permutations do exist, as we can show by exhibiting one; section~\ref{sec:antilinear} is devoted to their analysis.
\begin{lemma} \label{lem:antilinear:example}
 The permutation $\pi = (1234)$ is antilinear.
\end{lemma}
\begin{proof}
Clearly $\pi(0) = 0$. In order to verify the other condition, it is enough to show that for non-zero $x,y,z \in \ZZ_2^3$, $x \oplus y \oplus z = 0$ implies $\pi(x) \oplus \pi(y) \oplus \pi(z) \neq 0$. There are seven such triplets $x,y,z$, corresponding to the seven lines of the Fano plane:
\begin{align*}
 &\pi(2) \oplus \pi(4) \oplus \pi(6) = 3 \oplus 1 \oplus 6 = 4, \\
 &\pi(1) \oplus \pi(4) \oplus \pi(5) = 2 \oplus 1 \oplus 5 = 6, \\
 &\pi(3) \oplus \pi(4) \oplus \pi(7) = 4 \oplus 1 \oplus 7 = 2, \\
 &\pi(1) \oplus \pi(2) \oplus \pi(3) = 2 \oplus 3 \oplus 4 = 5, \\
 &\pi(2) \oplus \pi(5) \oplus \pi(7) = 3 \oplus 5 \oplus 7 = 1, \\
 &\pi(1) \oplus \pi(6) \oplus \pi(7) = 2 \oplus 6 \oplus 7 = 3, \\
 &\pi(3) \oplus \pi(5) \oplus \pi(6) = 4 \oplus 5 \oplus 6 = 7. \qedhere
\end{align*}
\end{proof}

We proceed to construct $V$.
\begin{lemma} \label{lem:V}
 There exists a three-dimensional subspace $V$ of $\Kvecsp{8}$ whose non-zero vectors are complements of cubes.
\end{lemma}
\begin{proof}
We construct a basis for $V$ in the shape of a $\ZZ_2^3$-coloring of the edges of $K_8$. Index the vertices of $K_8$ using $\ZZ_2^3$. The color $C(i,j)$ of the edge $(i,j)$ is defined by
\[ C(i,j) = \pi(i \oplus j). \]
The corresponding vector space $V$ is defined by
\[ V = \{ v_k : k \in \ZZ_2^3 \}, \quad v_k(i,j) = \inner{C(i,j)}{k}. \]

We proceed to show that $\comp{v_k}$ is a cube for $k \neq 0$; this also implies that $V$ is three-dimensional. Two vertices $(i,j)$ are connected in $\comp{v_k}$ if $\inner{\pi(i \oplus j)}{k} = 0$, or equivalently $i \oplus j \in \pi^{-1}(\orth{k})$. Thus every vertex $i$ is connected to the set of vertices $i \oplus N$, where
\[ N = \pi^{-1}(\orth{k} \setminus 0). \]
The three non-zero vectors $x,y,z \in \orth{k}$ satisfy $x \oplus y \oplus z = 0$, and so by antilinearity the three vectors in $N$ are linearly independent. A moment's reflection leads us to conclude that $\comp{v_k}$ is a cube.
\end{proof}

The proof of theorem~\ref{thm:main} follows the steps outlined above.
\begin{proof}[Proof of theorem~\ref{thm:main}]
 Recall that $\cF$ is a non-bipartite-agreeing family on $8$ vertices. The vector space $\Kvecsp{8}$ decomposes as a disjoint union of $2^{\binom{n}{2}-3}$ cosets of $V$. Suppose that $G_1,G_2$ belong to the same coset of $V$. Thus $G_1 \agr G_2 = \comp{G_1 \symdiff G_2} \in \comp{V}$, so that $G_1 \agr G_2$ is a cube. This implies that $\cF$ cannot contain both $G_1$ and $G_2$. Hence $\cF$ intersects each coset of $V$ in at most one graph, showing that $|\cF| \leq 2^{\binom{n}{2}-3}$.

If $|\cF| = 2^{\binom{n}{2}-3}$ then $\cF$ must intersect each coset of $V$ in exactly one graph. Suppose further that $\cF$ is non-bipartite-intersecting. If $\cF$ is not a triangle-junta then for each triangle $T$ it must contain $T \symdiff v_T$ for some non-zero $v_T \in V$. A computer search verifies that no choice of $v_T$ (for all $T$) results in a triangle-intersecting family. We conclude that $\cF$ must be a triangle-junta.

Finally, let $\cF$ be an arbitrary non-bipartite-agreeing family of size $2^{\binom{n}{2}-3}$. We sketch an argument showing that $\cF$ is a triangulumvirate; the full argument appears in~\cite{eff}. By applying a sequence of monotonizing operations, we transform $\cF$ into a monotone family $\cF_M$ which remains non-bipartite-agreeing. Since $\cF_M$ is monotone, it must be non-bipartite-intersecting. Hence it is a triangle-junta. By analyzing the process of monotonization (this is the difficult part), we conclude that $\cF$ must have been a triangulumvirate.
\end{proof}

\section{Generalizations} \label{sec:general}
Can the proof presented in the previous section be generalized for $n > 8$? Our proof relied on a linear subspace of $\Kvecsp{n}$ with certain properties. Unfortunately, in this section we show that the largest $n$ for which these properties hold is $8$. We next discuss several possible extensions of the proof.

We begin by showing that the proof of theorem~\ref{thm:main} cannot be extended literally for $n > 8$. The proof of theorem~\ref{thm:main} employs a three-dimensional subspace $V$ whose key property is that the complement of every non-empty graph in $V$ is bipartite. If we only want to bound the size of triangle-agreeing families, it is enough to demand that every non-empty graph is triangle-free. However, even this weaker property can hold only for $n \leq 8$.
\begin{proposition} \label{pro:main:noext}
 Suppose $V$ is a three-dimensional subspace of $\Kvecsp{n}$ such that the complement of every non-zero vector in $V$ is triangle-free. Then $n \leq 8$.
\end{proposition}
\begin{proof}
 In lemma~\ref{lem:V} we constructed such a subspace for $n=8$ using a coloring of the edges of $K_n$, which served as a basis for the subspace. We can reverse the process. Given a basis $b_0,b_1,b_2$ of $V$, we can color the edges of $K_n$ using elements of $\ZZ_2^3$ in a natural way:
\[ C(i,j) = (b_0(i,j),b_1(i,j),b_2(i,j)). \]
The different vectors in $V$ are then obtained according to the formula
\[ V = \{ v_k : k \in \ZZ_2^3 \}, \quad v_k(i,j) = \inner{C(i,j)}{k}. \]
We proceed to translate the assumption on $V$ to some property of the coloring. Suppose $x,y,z$ are the colors of some triangle in $K_n$. Then for all $k \neq 0$, not all of $\inner{x}{k},\inner{y}{k},\inner{z}{k}$ can be zero, for this would correspond to a triangle in $\comp{v_k}$. We can write this inequality in matrix form: thinking of $x,y,z$ as column vectors, let $M = \begin{pmatrix} x & y & z \end{pmatrix}$.
Then for all $k \neq 0$ we have $\tr{k}M \neq 0$, so that $M$ is regular. In particular, $x,y,z$ are all different and non-zero.

Any two edges incident to the same vertex can be completed to a triangle. Therefore all edges incident to a vertex must be colored using different non-zero colors. Since there are only $7$ non-zero elements in $\ZZ_2^3$, we deduce that $n \leq 8$.
\end{proof}

There are (at least) two natural ways to relax what we require of $V$:
\begin{enumerate}
 \item We can drop the assumption that $V$ is a linear subspace~\cite{ehudnotss}.
 \item Instead of requiring every non-zero element to be co-triangle-free, we can demand that every big enough subset of $V$ contains such an element; in order to get an upper bound of $2^{\binom{n}{2}-3}$, we will need the dimension of $V$ to be bigger as well.
\end{enumerate}

Taking both extensions into account, here is what we want of $V$.
\begin{definition} \label{def:useful}
 A subset $V$ of $\Kvecsp{n}$ of size $2^m$ is \emph{useful for triangles} if every subset $S$ of $V$ of size $2^{m-3}+1$ contains two vectors whose agreement is triangle-free.
\end{definition}

Using this definition, it is easy to extend the proof of theorem~\ref{thm:main}.
\begin{proposition} \label{pro:main:ext}
 Suppose there exists a subset $V$ of $\Kvecsp{n}$ which is useful for triangles. Then every triangle-agreeing family on $n$ vertices has size at most $2^{\binom{n}{2}-3}$.
\end{proposition}
\begin{proof}
 Let $\cF$ be a triangle-agreeing family. Just like in the proof of theorem~\ref{thm:main}, the properties of $V$ guarantee that $\cF$ intersects any coset of $V$ in at most $2^{m-3}$ vectors. We can choose a random vector in $\Kvecsp{n}$ uniformly by first choosing a random coset of $V$ and then choosing a random point in the coset. The probability that this random point is in $\cF$ is at most $1/8$, and so $\cF$ contains at most $1/8$ of the total number of graphs.
\end{proof}

We have reduced the problem of generalizing the proof of~\ref{thm:main} to that of construction a subset which is useful for triangles. Adapting the proof of proposition~\ref{pro:main:noext}, we can show that the size of $V$ must grow with $n$.
\begin{proposition} \label{pro:main:ext:dim}
 Suppose $V \subseteq \Kvecsp{n}$ is useful for triangles. For every $m$ there exists a constant $N_m$ such that $|V| \leq 2^m$ implies $n \leq N_m$.
\end{proposition}
\begin{proof}
 Let $V = \{v_k\}$. We color the edges of $K_n$ using colors from $\ZZ_2^{|V|}$ in the natural way:
\[ C(i,j) = (v_1(i,j),\ldots,v_{|V|}(i,j)). \]
Ramsey's theorem supplies us with some $N_m$ such that if $n > N_m$ then there exists some monochromatic triangle $T$ in the coloring, say colored by $c$. The most common bit in $c$ is shared by at least $|V|/2$ coordinates. Thus there are at least $|V|/2$ vectors in $V$ whose agreement contains $T$. Therefore $V$ isn't useful for triangles.
\end{proof}

It turns out that we can narrow our focus to those subsets which are in fact linear subspaces.
\begin{proposition} \label{pro:useful:span}
 Suppose $S \subseteq \Kvecsp{n}$ is useful for triangles. Its linear span $V$ is also useful for triangles.
\end{proposition}
\begin{proof}
 Let $T$ be a subset of $V$ of size larger than $|V|/8$. Consider the following process for picking a uniform random element of $V$: pick uniformly at random $v \in V$, and pick a random element of the coset $v \symdiff S$. The probability that the resulting element lies in $T$ is larger than $1/8$, and we conclude that there is a coset $v \symdiff S$ such that $|(v \symdiff S) \cap T| > |v \symdiff S|/8$. Equivalently, $|S \cap (v \symdiff T)| > |S|/8$. Since $S$ is useful for triangles, there are two elements $s_1,s_2 \in S \cap (v \symdiff T)$ whose agreement is triangle-free. Since $(s_1 \symdiff v) \agr (s_2 \symdiff v) = s_1 \agr s_2$, we deduce that there are two elements $s_1\symdiff v, s_2\symdiff v$ in $T$ whose agreement is triangle-free. Thus $V$ is useful for triangles.
\end{proof}

We have so far been unable to construct subsets of $\Kvecsp{n}$ useful for triangles even for $n=9$, and we question their existence. The most we can show is that there exist no four-dimensional subspaces of $\Kvecsp{9}$ which are useful for triangles.
\begin{proposition} \label{pro:useful:no4d}
 There exist no four-dimensional subspaces of $\Kvecsp{9}$ which are useful for triangles.
\end{proposition}
\begin{proof}[Proof sketch]
Suppose $V$ is a four-dimensional subspace of $\Kvecsp{n}$ which is useful for triangles. Thus, for any three different vectors $x,y,z \in V$, one of $x \agr y, x \agr z, y \agr z$ is triangle-free. What sort of subspaces give rise to such a property? Picking a basis of $V$, we can naturally write $V = \{ v_i : i \in \ZZ_2^4 \}$ in such a way that $v_i \symdiff v_j = v_{i \oplus j}$. Let
\[ I = \{ i \in \ZZ_2^4 : \comp{v_i}\text{ is triangle-free}\}. \]
The property of usefulness for triangles translates into the following property of $I$: if $x \neq y \notin I$ then $x \oplus y \in I$. Using this property, an elementary argument shows that either $I$ contains a three-dimensional subspace, or $I$ is equivalent (up to a linear mapping) to the set
\[ I_0 = \{i \in \ZZ_2^4 : |i| \in \{1,2\}\}. \]
In the first case, $V$ contains a three-dimensional subspace which is useful for triangles, so proposition~\ref{pro:main:noext} shows that $n \leq 8$. In the second case, using the constraints implied by $I_0$ we can search for a solution for $n=9$ recursively. The computer search comes up with no solutions.
\end{proof}

We are left with the following open question: do there exist subspace of $\Kvecsp{n}$ useful for triangles for $n > 8$?

\section{Antilinear permutations} \label{sec:antilinear}
Lemma~\ref{lem:antilinear:example} gives one example of an antilinear permutation. In this section we explore some of the properties of antilinear permutations. These properties enable us to describe and enumerate all $1344$ antilinear permutations in terms of the much smaller class of eight \emph{Fano permutations}.

Our starting point is the proof of lemma~\ref{lem:antilinear:example}. Observing the proof, we are led to the following definition.
\begin{definition} \label{def:signature}
 Let $\pi \in S_8$. Its \emph{signature} $\sigma \colon \ZZ_2^3 \longrightarrow \ZZ_2^3$ is defined by
\[ \sigma(x) = \sum_{y \in \orth{x}} \pi(y). \]
\end{definition}
Note that $\sigma(0) = 0$ always. In the case of lemma~\ref{lem:antilinear:example}, the signature was a permutation, and this is no coincidence. In fact, it is even linear.

\begin{lemma} \label{lem:signature}
If $\pi \in S_8$ satisfies $\pi(0) = 0$ then its signature $\sigma$ is a linear transformation. If $\pi$ is antilinear then $\sigma$ is regular and so a permutation.
\end{lemma}
\begin{proof}
 In order to show that $\sigma$ is a linear transformation, we pick any non-zero $x,y,z$ summing to zero, and show that their images under $\sigma$ also sum to zero. We do this by expanding the definition of $\sigma$ and cancelling like terms. Write
\[ \sigma(s) = \sum_t \delta(s,t) \pi(t), \quad \delta(s,t) = \begin{cases} 1 & \text{if } \inner{s}{t} = 0, \\ 0 & \text{otherwise}. \end{cases} \]
We can now expand $\sigma(x)\xor\sigma(y)\xor\sigma(z)$:
\[ \sigma(x)\xor\sigma(y)\xor\sigma(z) = \bigxor_w c_w \pi(w), \quad c_w=\delta(x,w)+\delta(y,w)+\delta(z,w). \]
Since $x\xor y\xor z = 0$, there must be some $u \in \ZZ_2^3$ orthogonal to all of them. Notice that
\[ c_w = |\orth{u} \cap \orth{w}| - 1. \]
We deduce that $c_w \in \{1,3\}$, and so
\[ \sigma(x)\xor\sigma(y)\xor\sigma(z) = \bigxor_w \pi(w) = 0. \]
Thus $\sigma$ is a linear transformation.

If $\pi$ is antilinear then $\sigma(x) \neq 0$ for $x \neq 0$ by antilinearity, so that $\sigma$ is regular.
\end{proof}

Applying a linear transformation to any antilinear permutation, we can always reach a permutation whose signature is the identity permutation. We term this class of permutations \emph{Fano permutations}.
\begin{definition} \label{def:fano}
 An antilinear permutation $\pi \in S_8$ is \emph{Fano} if its signature is the identity.
\end{definition}

\begin{lemma} \label{lem:fano:normal:form}
 Every antilinear permutation has a unique representation of the form $L \varphi$, where $L$ is linear and $\varphi$ is Fano.
\end{lemma}
\begin{proof}
 Suppose $\pi$ is an antilinear permutation with signature $L$. Lemma~\ref{lem:signature} implies that $L$ is a regular linear transformation, and so $\varphi = L^{-1} \pi$ is also antilinear. It is easy to see that the signature of $\varphi$ is $L^{-1} L = I$, so that $\varphi$ is Fano. Thus $\pi = L \varphi$ is the required representation.

 Conversely, notice that $L$ is the signature of $\pi = L \varphi$, so that $\pi$ determines $L$ (given that $\varphi$ is Fano). Since $L$ is regular, $\pi$ and $L$ determine $\varphi$.
\end{proof}

We continue with two simple properties of Fano permutations.
\begin{lemma} \label{lem:fano:orth1}
 Let $\pi$ be a Fano permutation. If $x$ is non-zero then
\[ \inner{x}{\pi(x)} = 1. \]
\end{lemma}
\begin{proof}
 Suppose to the contrary that $x \in \orth{\pi(x)}$ for some non-zero $x$. Let $\orth{\pi(x)} = \{0,x,y,z\}$. The Fano property implies that
\[ \pi(x) \xor \pi(y) \xor \pi(z) = \pi(x), \]
so that $\pi(y) = \pi(z)$. This contradicts the fact that $\pi$ is a permutation.
\end{proof}

\begin{lemma} \label{lem:fano:orth2}
 Let $\pi$ be a Fano permutation. If $x \neq y$ are non-zero then
\[ \inner{x}{\pi(y)} \xor \inner{y}{\pi(x)} = 1. \]
\end{lemma}
\begin{proof}
 Denote by $z$ the unique non-zero element orthogonal to both $x$ and $y$, so that $\orth{z} = \{0,x,y,x\xor y\}$. Lemma~\ref{lem:fano:orth1} implies that
\begin{align*}
 1 &= \inner{x\xor y}{\pi(x\xor y)} = \inner{x\xor y}{\pi(x)\xor\pi(y)\xor z} \\ &=
 \inner{x\xor y}{\pi(x)} \xor \inner{x\xor y}{\pi(y)} = \inner{y}{\pi(x)} \xor \inner{x}{\pi(y)}. \qedhere
\end{align*}
\end{proof}

The preceding two properties enable us to enumerate all Fano permutations by hand.
\begin{lemma} \label{lem:fano}
 There are eight Fano permutations:
\begin{align*}
 &(135647) &(174652) \\
 &(153627) &(172635) \\
 &(236547) &(274563) \\
 &(13)(26)(45) &(15)(23)(46)
\end{align*}
\end{lemma}
\begin{proof}
 A Fano permutation $\pi$ is determined by $\pi(1),\pi(2),\pi(4)$ since
\begin{align*}
 \pi(3) &= \pi(1) \xor \pi(2) \xor 4, \\
 \pi(5) &= \pi(1) \xor \pi(4) \xor 2, \\
 \pi(6) &= \pi(2) \xor \pi(4) \xor 1, \\
 \pi(7) &= \pi(1) \xor \pi(2) \xor \pi(4) \xor 7.
\end{align*}
Since $\inner{1}{\pi(1)} = 1$, we know that $\pi(1) \in \{1,3,5,7\}$. Similarly, $\inner{2}{\pi(2)} = 1$ and so $\pi(2) \in \{2,3,6,7\}$. Moreover, $\inner{1}{\pi(2)} = \inner{2}{\pi(1)} \xor 1$. Thus if $\pi(1) \in \{1,5\}$ then $\pi(2) \in \{3,7\}$, and if $\pi(1) \in \{3,7\}$ then $\pi(2) \in \{2,6\}$. Given any choice of $\pi(1),\pi(2)$, similar conditions determine $\pi(4)$ and so the rest of the permutation.
\end{proof}

Combining lemma~\ref{lem:fano} with the earlier lemma~\ref{lem:fano:normal:form}, we can enumerate all antilinear permutations.
\begin{corollary} \label{cor:antilinear:enum}
 There are $8(8-1)(8-2)(8-4) = 1344$ antilinear permutations.
\end{corollary}
\begin{proof}
Lemma~\ref{lem:fano:normal:form} shows that each antilinear permutation is obtained uniquely by multiplying a regular linear transformation and a Fano permutation. There are $(8-1)(8-2)(8-4)$ of the former and $8$ of the latter. 
\end{proof}

\bibliography{Cubes}{}
\bibliographystyle{plain}

\end{document}